\documentclass[12pt]{amsart}

\usepackage{latexsym}
\usepackage{psfrag}
\usepackage{amsmath}
\usepackage{amssymb}
\usepackage{epsfig}
\usepackage{amsfonts}
\usepackage{amscd}
\usepackage{mathrsfs}
\usepackage{graphicx}
\usepackage{enumerate}

\newtheorem{theorem}{Theorem}
\newtheorem{lemma}[theorem]{Lemma}

\newtheorem{coro}[theorem]{Corollary}

\newtheorem{conj}[theorem]{Conjecture}

\begin{document}

\title[Radial Limits]{Radial Limits of Partial Theta and Similar Series}

\author[Kur\c{s}ung\"{o}z]{Ka\u{g}an Kur\c{s}ung\"{o}z}
\address{Faculty of Engineering and Natural Sciences, Sabanc{\i} University, \.{I}stanbul, Turkey}
\email{kursungoz@sabanciuniv.edu}



\date{April 2015}

\begin{abstract}
\noindent
We study unilateral series in a single variable $q$ where its exponent is an unbounded increasing function, 
and the coefficients are periodic.  
Such series converge inside the unit disk.  
Quadratic polynomials in the exponent correspond to partial theta series.  
We compute limits of those series as the variable tends radially to a root of unity.  
The proofs use ideas from the $q$-integral and are elementary.  
\end{abstract}

\maketitle


\section{Introduction}
\label{secIntro}

Consider the series 
\begin{align*}
  \sum_{n \geq 0} (-1)^n q^n
  = 1 - q + q^2 - q^3 + \cdots
\end{align*}
which converges for $\vert q \vert < 1$.  
This is a geometric series, so
\begin{align*}
  \sum_{n \geq 0} (-1)^n q^n
  = \frac{1}{1 + q}.  
\end{align*}
We can easily compute 
\begin{align*}
  \lim_{q \to 1^-} \sum_{n \geq 0} (-1)^n q^n
  = \lim_{q \to 1^-} \frac{1}{1 + q} 
  = \frac{1}{2}.  
\end{align*}

In fact, we can quantify the rate of convergence, as well.  
By letting $q = e^{-x}$, $q \to 1^-$ corresponds to $x \to 0^+$, and 
\begin{align*}
  \frac{1}{1 + q}
  = \frac{1}{ 1 + e^{-x}} 
  =: f(x).  
\end{align*}
Because
\begin{align*}
  f'(0) = & (-1)(1 + e^{-x})^{-2}e^{-x} \vert_{x = 0} = - \frac{1}{4},  \\ 
  f''(0) = & 2(1 + e^{-x})^{-3}e^{-x} + (1 + e^{-x})^{-2} e^{-x} \vert_{x = 0} = \frac{1}{2}, \\
  & \vdots
\end{align*}
\begin{align*}
  f(x) \approx \frac{1}{2} - \frac{x}{4} + \frac{x^2}{4} + \cdots
\end{align*}
as $x \to 0^+$,  or 
\begin{align}
\label{eqGeomAsymp}
  \sum_{n \geq 0} (-1)^n q^n \approx \frac{1}{2} - \frac{(- \log q)}{4} + \frac{(- \log q)^2}{4} + \cdots
\end{align}
as $q \to 1^-$.  
In other words, 
\begin{align*}
  \left( \sum_{n \geq 0} (-1)^n q^n\right) - \frac{1}{2} 
  = O(-\log q), \quad \textrm{ and} \quad 
  \lim_{q \to 1^-} \frac{\left( \sum_{n \geq 0} (-1)^n q^n\right) - \frac{1}{2}}{ -\log q}
  = - \frac{1}{4}, 
\end{align*}
and 
\begin{align*}
  & \left( \sum_{n \geq 0} (-1)^n q^n\right) - \frac{1}{2} + \frac{(-\log q)}{4} 
  = O( (-\log q)^2 ), \\ \\ & \textrm{ and} \quad 
  \lim_{q \to 1^-} \frac{\left( \sum_{n \geq 0} (-1)^n q^n\right) - \frac{1}{2} + \frac{(-\log q)}{4} }
		{ (-\log q)^2 }
  = \frac{1}{2}, 
\end{align*}
and so on.  
The right hand side of \eqref{eqGeomAsymp} is called the asymptotic expansion 
of the series on the left hand side as $q \to 1^-$ \cite[Ch. 15]{RNBII}.  

The problem already becomes harder when we try to do the same for a series such as 
\begin{align}
\label{eqPartialTheta}
  \sum_{n \geq 0} (-1)^n q^{n^2}.  
\end{align}
This unilateral sum is called a partial theta series, since its bilateral version is a theta series.  
We cannot get away with merely taking successive derivatives, 
since we do not have a closed form of \eqref{eqPartialTheta}.  
One solution is to use Euler's integral formula \cite[Ch. 13, eq. (10.5)]{RNBII}. 
Coefficients in the asymptotic expansions of both \eqref{eqGeomAsymp} and \eqref{eqPartialTheta}
are explicitly given in \cite{LZ}.  
In particular, 
\begin{align*}
 \lim_{q \to 1^-} \sum_{n \geq 0} (-1)^n q^{n^2}
 = \frac{1}{2}.  
\end{align*}
Their method is to use Euler's integral formula or computing certain contour integrals.  

We will elementarily prove that 
\begin{align*}
 \lim_{q \to 1^-} \sum_{n \geq 0} (-1)^n q^{s(n)}
 = \frac{1}{2} 
\end{align*}
for any polynomial $s(n)$ with arbitrary positive degree and positive leading coefficient, 
borrowing ideas from the $q$-integral \cite[Sec. 10.1]{AAR}, \cite[Sec 1.11]{GR}.  

We will also show that this phenomenon does not necessarily occur for functions $s(n)$
with exponential growth.  For instance, we will prove that 
\begin{align*}
 \lim_{q \to 1^-} \sum_{n \geq 0} (-1)^n q^{a^n}
\end{align*}
cannot exist for large enough $a$.  
This is interesting in the sense that replacing $a^n$ 
with the Taylor polynomial of any degree approximating it (in $n$)
will yield $1/2$ for this limit.  

For asymptotic expansions, the $q$-integral, unfortunately, does not readily help.  
We have to resort to Euler's integral formula, 
or Mellin transform and computing contour integrals \cite[Ch. 15]{RNBII}.  

The paper is organized as follows.  
We introduce the idea of the $q$-integral in Section \ref{secPrelim}, 
present the main results in Section \ref{secMain}, 
give some applications in Section \ref{secRadial}, 
and finally discuss asymptotic expansions in Section \ref{secAsymp}.  

\section{Preliminaries}
\label{secPrelim}

Suppose $f:[0, 1] \to [0, 1]$ is a continuous, therefore Riemann integrable function.  
To compute $\int_0^1 f(x) \mathrm{d}x$, 
instead of a finite subdivision of the domain $[0, 1]$, 
we can consider the subintervals
\begin{align*}
  \ldots, \quad [q^3, q^2], \quad [q^2, q], \quad [q, 1]
\end{align*}
for $q \in (0, 1)$.  
Using the right-hand endpoints of these intervals, 
we can write an approximating sum for our integral.  
\begin{align*}
  \int_0^1 f(x) \, \mathrm{d} x 
  \approx \sum_{n \geq 0} f(q^n) (q^n - q^{n+1})
  = (1 - q) \sum_{n \geq 0} f(q^n) q^n
\end{align*}

\begin{center}
 \includegraphics[scale=0.5]{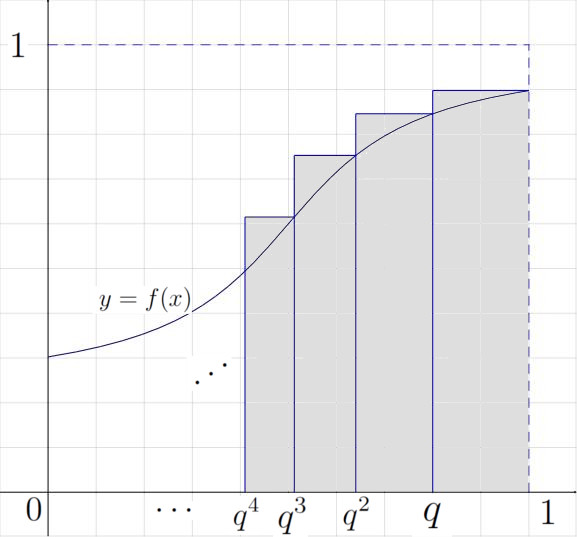}
\end{center}

As $q \to 1^{-}$, the approximations get more accurate, and
\begin{align*}
  \int_0^1 f(x) \, \mathrm{d} x 
  = \lim_{q \to 1^{-}} (1 - q) \sum_{n \geq 0} f(q^n) q^n
\end{align*}

For instance, let $c$ be any positive real number, and $f(x) = x^c$ for $0 < x < 1$.  
\begin{align*}
  \int_0^1 x^c \, \mathrm{d} x 
  = \lim_{q \to 1^{-}} (1 - q) \sum_{n \geq 0} (q^{c+1})^n
  = \lim_{q \to 1^{-}} \frac{1 - q}{ 1 - q^{c+1} } 
  = \frac{1}{c+1}, 
\end{align*}
using L'H\^{o}spital's rule.  

This is called the $q$-integral.  
It works in greater generality, and can handle improper integrals as well.  
More information on the $q$-integral can be found in \cite[Sec. 10.1]{AAR} or \cite[Sec 1.11]{GR}.  
We will be content with sums approximating integrals of continuous functions $f:[0, 1] \to [0, 1]$.  

\section{Main Results}
\label{secMain}

We assume $q \in (0, 1)$ throughout this section.  

\begin{lemma}
\label{lemmaMain}
  Let $x(t)$, $y(t)$ be eventually increasing real functions such that 
  \begin{equation*}
    \lim_{t \to \infty} x(t)
    = \lim_{t \to \infty} y(t)
    = \infty, 
    \qquad
    \lim_{t \to \infty} \frac{y(t-1)}{y(t)}
    = 1, 
    \quad \textrm{and } \quad 
    \lim_{t \to \infty} \frac{x(t)}{y(t)} = c  
  \end{equation*}
  for some $0 < c < \infty$.  Then
  \begin{equation*}
    \lim_{q \to 1^{-}} \sum_{n \geq 0} q^{y(n)} \left( q^{x(n)} - q^{x(n+1)} \right)
    = \frac{1}{1+c}.  
  \end{equation*}
\end{lemma}

\begin{proof}
  Determine $N \in \mathbb{N}$ such that both $x(t)$ and $y(t)$ are positive and increasing for $t \geq N$.  
  Then for all $q \in (0, 1)$, and $n \geq N+1$, $q^{y(n)} \left( q^{x(n)} - q^{x(n+1)} \right)$
  is the area of the rectangle $R(n)$ with vertices 
  $(q^{x(n)}, 0)$, $(q^{x(n+1)}, 0)$, $(q^{x(n)}, q^{y(n)})$ and $(q^{x(n+1)}, q^{y(n)})$.  
  
  \begin{center}
    \includegraphics[scale=0.5]{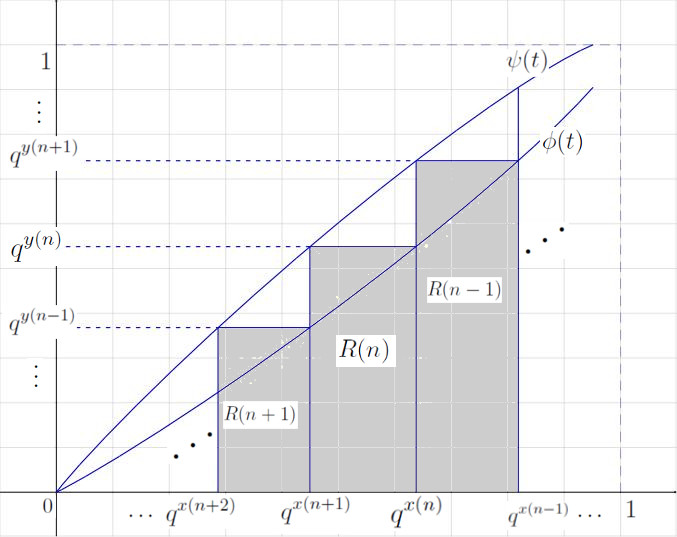}
  \end{center}
  
  Define $\varphi(t) = (q^{x(t)}, q^{y(t)})$, and $\psi(t) = (q^{x(t)}, q^{y(t-1)})$ for $t\geq N+1$.
  Then, for all $n \geq N+1$
  \begin{equation*}
    q^{y(t)} \leq q^{y(n)} \leq q^{y(t-1)}
  \end{equation*}
  when $n \leq t \leq n+1$.  
  In other words, the curve defined by $\varphi(t)$ is below the top side of the rectangle $R(n)$, 
  and the curve defined by $\psi(t)$ is above it for $n \leq t \leq n+1$.  Thus, 
  \begin{equation*}
    \int_{t = n}^{n+1} q^{y(t)} \; \mathrm{d} \, q^{x(t)}
    \leq q^{y(n)} \left( q^{x(n)} - q^{x(n+1)} \right) 
    \leq \int_{t = n}^{n+1} q^{y(t-1)} \; \mathrm{d} \, q^{x(t)}
  \end{equation*}
  for $n \geq N+1$.  Summing over those $n$, we find 
  \begin{equation*}
    \int_{t = N+1}^{\infty} q^{y(t)} \; \mathrm{d} \, q^{x(t)}
    \leq \sum_{n \geq N+1} q^{y(n)} \left( q^{x(n)} - q^{x(n+1)} \right) 
    \leq \int_{t = N+1}^{\infty} q^{y(t-1)} \; \mathrm{d} \, q^{x(t)}
  \end{equation*}
  The integrals on either end are convergent because both curves 
  $\varphi(t)$ and $\psi(t)$ determine functions 
  the graph of which are inside the unit square $[0, 1] \times [0, 1]$ for $t \geq N+1$.  
  
  Now, for an arbitrary but fixed $x_0 \in (0, 1)$, 
  if $x_0 = q^{x(t)}$, then $ x(t) = { (- \log x_0) }/{ (- \log q) }$.  
  Because $ \lim_{q \to 1^{-}} { (- \log x_0) }/{ (- \log q) } = \infty$, 
  $x(t)$ can be made as large as possible, 
  by choosing $q$ close enough to 1.  
  Then $t$ is uniquely determined because $x(t)$ is increasing for large enough $t$.  
  For the same $t$, let $y_0 = q^{y(t)}$.  
  Then, $y(t) = (-\log y_0)/(-\log q)$, and hence $x(t)/y(t) = (-\log y_0)/(-\log x_0)$.  
  Because $\lim_{t \to \infty} x(t) / y(t) = c$, 
  and $t \to \infty$ as $q \to 1^{-}$, $y_0 \sim x_0^c$ as $q \to 1^{-}$.  
  
  Similar computations for $y_1 = q^{y(t-1)}$ for $t$ as determined in the previous paragraph bring
  $y_1 \sim x_0^c$ as $q \to 1^{-}$ thanks to the constraint $\lim_{t \to \infty} y(t)/y(t-1) = 1$.  
  
  Consequently, both $\varphi(t)$ and $\psi(t)$ converge pointwise to $y = x^c$ for $0 < x < 1$.  
  Lebesgue's dominated convergence theorem \cite[Theorem 1.34]{Rudin} yields
  \begin{equation*}
    \lim_{q \to 1^{-}} \int_{t = N+1}^{\infty} q^{y(t)} \; \mathrm{d} \, q^{x(t)}
    = \lim_{q \to 1^{-}} \int_{t = N+1}^{\infty} q^{y(t-1)} \; \mathrm{d} \, q^{x(t)}
    = \int_0^1 x^c \mathrm{d} \, x
    = \frac{1}{1+c}.  
  \end{equation*}

  Finally, the squeeze theorem (or the sandwich lemma) \cite[p.68]{Stewart} ensures 
  \begin{equation*}
    \lim_{q \to 1^{-}} \sum_{n \geq N+1} q^{y(n)} \left( q^{x(n)} - q^{x(n+1)} \right) 
    = \frac{1}{1+c}.  
  \end{equation*}
  The proof is complete once we observe that 
  \begin{equation*}
    \lim_{q \to 1^{-}} \sum_{n = 0}^N q^{y(n)} \left( q^{x(n)} - q^{x(n+1)} \right) = 0.  
  \end{equation*}
\end{proof}

Although $\lim_{q^{-}} \sum_{n = 0}^\infty q^{y(n)} \left( q^{x(n)} - q^{x(n+1)} \right)$ 
is not exactly the $q$-integral of the function $y = x^c$ for $0 < x < 1$, 
it is interpreted as an approximating sum.  
At least the spirit of the proof is the $q$-integral.  

Notice that the conditions for Lemma \ref{lemmaMain} are satisfied by any pair of real polynomials
$x(t)$ and $y(t)$ with the same positive degree and positive leading coefficients, 
and they are not satisfied if $y(t)$ is a function with exponential growth.  
Armed with these observations, let us draw some conclusions.  

\begin{theorem}
\label{thmPeriodicMeanZeroCoeffs}
  Suppose $s(n)$ is a real polynomial with positive degree $d$ 
  and positive lading coefficient.  
  Let $C: \mathbb{N} \to \mathbb{C}$ be a periodic function 
  with period $k$ and mean zero.  
  (i.e. $C(n) = C(n+k)$ for all $n \in \mathbb{N}$, 
  and $C(1) + \cdots + C(k) = 0$).  
  Then
  \begin{equation*}
    \lim_{q \to 1^{-}} \sum_{n \geq 0} C(n) q^{s(n)} 
    = \frac{-C(1) - 2C(2) - \cdots - (k-1)C(k-1)}{k}.  
  \end{equation*}
\end{theorem}
If $k$ is a period, then so are $2k$, $3k$ etc.  
The reader can readily verify that the right hand side of the limit is well defined.  

\begin{proof}
 We can rewrite the sum
 \begin{align}
 \label{eqPeriodicDecompWRTCoeff}
  \sum_{n \geq 0} C(n) q^{s(n)}
  & = \sum_{j = 0}^{k-1} C(j) \sum_{n \geq 0} q^{s(nk+j)} \\
  \nonumber
  & = \left( \sum_{j = 0}^{k-2} C(j) \sum_{n \geq 0} q^{s(nk+j)} \right)
    + C(k-1) \sum_{n \geq 0} q^{s(nk+k-1)} \\
  \nonumber
  & = \sum_{j = 0}^{k-2} C(j) \sum_{n \geq 0} q^{s(nk+j)} - q^{s(nk+k-1)}
 \end{align}
 since $C(0) + \cdots + C(k-1) = 0$.  
 
 The next step is to decompose 
 \begin{equation*}
  \sum_{n \geq 0} q^{s(nk+j)} - q^{s(nk+k-1)}
  = \sum_{n \geq 0} q^{y(n)} \left( q^{x(n)} - q^{x(n+1)} \right).  
 \end{equation*}
 for arbitrary but fixed $j = 0, 1, \ldots, k-2$.  
 Apparently
 \begin{align}
 \label{eqPeriodicDecompXY2}
  x(n) + y(n) = & s(nk + j) \\
  \nonumber
  x(n+1) + y(n) = & s(nk + k-1).  
 \end{align}
 Subtracting the first equation from the second, we have
 \begin{equation*}
  x(n+1) - x(n) = s(nk + k-1) - s(nk + j).  
 \end{equation*}
 For convenience, we assume $x(0) = 0$ 
 and add instances of the last equation 
 for $n-1, n-2, \ldots, 1$ side by side to obtain
 \begin{equation}
 \label{eqPeriodicDecompXSum}
  x(n) = \sum_{l = 0}^{n-1} s(lk + k-1) - s(lk + j).  
 \end{equation}
 
 On the other hand, if $s(n) = a_d n^d + a_{d-1} n^{d-1} + O(n^{d-2})$, then
 \begin{align*}
  s(nk+j) = & a_d (nk+j)^d + a_{d-1}(nk+j)^{d-1} + O(n^{d-2}) \\
  = & a_d k^d n^d + d a_d k^{d-1} j n^{d-1} + a_{d-1} k^{d-1} n^{d-1} + O(n^{d-2})
 \end{align*}
 by the binomial theorem.  
 Thus, 
 \begin{equation*}
  s(nk + k - 1) - s(nk + j) 
  = d a_d k^{d-1} (k-1-j) n^{d-1} + O(n^{d-2}).  
 \end{equation*}
 Since $\sum_{l = 0}^{n-1} l^r = l^{r+1}/(r+1) + O(l^r)$ \cite[p.107]{CG}, 
 it follows from \eqref{eqPeriodicDecompXSum} that 
 \begin{equation*}
  x(n) = a_d k^{d-1} (k-1-j) n^d + O(n^{d-1}), 
 \end{equation*}
 and combining this with \eqref{eqPeriodicDecompXY2} that 
 \begin{equation*}
  y(n) = a_d k^{d-1} (j+1) n^d + O(n^{d-1}).  
 \end{equation*}
 In particular, both $x(n)$ and $y(n)$ are real polynomials of the same degree as $s(n)$
 with positive leading coefficients.  
 They satisfy the hypotheses of Lemma \ref{lemmaMain} with 
 \begin{equation*}
  \lim_{t \to \infty} \frac{y(t)}{x(t)} = \frac{j+1}{k-1-j}, 
 \end{equation*}
 therefore by Lemma \ref{lemmaMain}
 \begin{equation*}
  \lim_{q \to 1^{-}} \sum_{n \geq 0} q^{s(nk+j)} - q^{s(nk+k-1)}
  = \frac{k-1-j}{k}, 
 \end{equation*}
 and by \eqref{eqPeriodicDecompWRTCoeff}
 \begin{align*}
  \lim_{q \to 1^{-}} \sum_{n \geq 0} C(n) q^{s(n)} 
  = & \frac{(k-1)C(0) + (k-2)C(1) + \cdots + 2 C(k-2)}{k} \\
  = & \frac{-C(1) - 2C(2) - \cdots - (k-1) C(k-1)}{k}.  
 \end{align*}
\end{proof}

\begin{coro}
\label{coroAlternating}
  Suppose $s(n)$ is a real polynomial with positive degree and positive leading coefficient.  Then, 
  \begin{align*}
    \lim_{q \to 1^{-}} \sum_{n \geq 0} (-1)^n q^{s(n)} = \frac{1}{2}.  
  \end{align*}
\end{coro}

This limit previously appeared in literature for $s(n)$ having degrees 1 and 2, 
or $s(n)$ being a monomial with arbitrary degree.  
It appeared in \cite[eq. (5.42)]{BMO} 
for $s(n) = $ quadratic polynomials with constant term zero.  
In fact, Theorem \ref{thmPeriodicMeanZeroCoeffs} for $s(n) = $ specific linear or quadratic polynomials 
are corollaries of one of Lawrence and Zagier's results \cite[p.98]{LZ}.  
Theorem \ref{thmPeriodicMeanZeroCoeffs} is also a corollary of one of Ramanujan's formulas
for $s(n) = $ a monomial with arbitrary degree \cite[Ch 15, Theorem 3.1]{RNBII}.  
However, we did not come across the case $s(n) = $ an arbitrary polynomial in the literature.  

\begin{theorem}
\label{thmPeriodicExpoExpo}
  Let $C:\mathbb{N} \to \mathbb{C}$ be a periodic function with period $k$ and mean zero.  
  Let $a > 1$ be a real number.  Then 
  \begin{equation*}
    \lim_{q \to 1^{-}} \sum_{n \geq 0} C(n) q^{a^n}
  \end{equation*}
  does not converge for large enough $a$.  
\end{theorem}

\begin{proof}
 We decompose the sum as in the proof of Lemma \ref{lemmaMain}.  
 \begin{align*}
  \sum_{n \geq 0} C(n) q^{a^{n}}
  = & \sum_{j = 0}^{k-1} C(j) \sum_{n \geq 0} q^{a^{nk+j}} \\
  = & \sum_{j = 0}^{k-2} C(j) \sum_{n \geq 0} q^{a^{nk+j}} - q^{a^{nk+k-1}} \\
  = & \sum_{j = 0}^{k-2} C(j) \sum_{n \geq 0} q^{y_j(n)} \left( q^{x_j(n)} - q^{x_j(n+1)} \right)
 \end{align*}
 If we assume $x_j(0) = (a^{k-1} - a^{j}) /(a^k - 1)$ for convenience, then 
 \begin{equation*}
  x_j(n) = a^{nk}\frac{ (a^{k-1} - a^{j} ) }{ (a^k-1) }
  \qquad \textrm{and } \qquad
  y_j(n) = \frac{a^{nk} (a^{k+j} - a^{k-1} ) }{ (a^k-1) }
 \end{equation*}
 are uniquely determined.  
 Let us note that $y_j(n-1) = \frac{a^{nk} (a^{j} - a^{-1} ) }{ (a^k-1) }$.  
 Also set
 \begin{equation*}
  M(a) = \frac{ a^{k+j} - a^{k-1} }{ a^{k-1} - a^{j} }
  \qquad \textrm{and } \qquad
  m(a) = \frac{ a^{j} - a^{-1} }{ a^{k-1} - a^{j} }
 \end{equation*}
 for some fixed $j$.  
 
 The change of variable $q \leftarrow q^{(a^k - 1)/(a^{k-1} - a^{j})}$ will not change the limit because 
 $q \to 1^{-}$ and $q^{(a^k - 1)/(a^{k-1} - a^{j})} \to 1^{-}$ simultaneously, for fixed $a$.  
 
 For arbitrary but fixed $j = 0, 1, \ldots, k-2$, 
 we will first demonstrate the failure of 
 \begin{equation*}
  \lim_{q \to 1^{-}} \sum_{n \geq 0} q^{a^{kn} M(a)} \left( q^{a^{nk}} - q^{a^{nk+k}} \right)
 \end{equation*}
 to converge for large enough $a$.  
 
 Again, as in the proof of Lemma \ref{lemmaMain}, 
 the term $q^{a^{kn} M(a)} ( q^{a^{nk}} - q^{a^{nk+k}} )$ for any $n \geq 0$
 signifies the area of the rectangle $R_n$ with vertices 
 $(q^{ a^{nk} }, 0)$, $(q^{ a^{nk+k} }, 0)$, $(q^{ a^{nk} }, q^{ a^{nk} M(a) })$, and 
 $(q^{ a^{nk+k} }, q^{ a^{nk} M(a) })$ $= (q^{ a^{nk+k} }, q^{ a^{nk + k} m(a) })$.  
 
 Define $\varphi(t) = ( q^{ a^{kt} }, q^{ a^{kt} M(a) } )$, and 
 $\psi(t) = ( q^{ a^{kt} }, q^{ a^{kt} m(a) } )$ for $t \geq 0$.  
 Notice that $\varphi(t)$ coincides with the graph of $y = x^{M(a)}$, 
 and $\psi(t)$ with the graph of $y = x^{m(a)}$ for $0 < x \leq q$.  
 Consequently, the top right corners of all rectangles $R_n$ are on the curve $y = x^{M(a)}$, 
 and the top left corners are on $y = x^{m(a)}$.  
 
 For $a > 1$, $m(a) < 1 < M(a)$, so that $y = x^{M(a)}$ is below the main diagonal $x = y$, 
 and $y = x^{m(a)}$ is above it when $0 < x < 1$.  
 Moreover, $\lim_{a \to \infty} M(a) = \infty$ and $\lim_{a \to \infty} m(a) = 0$.  
 
 \begin{center}
  \includegraphics[scale=0.5]{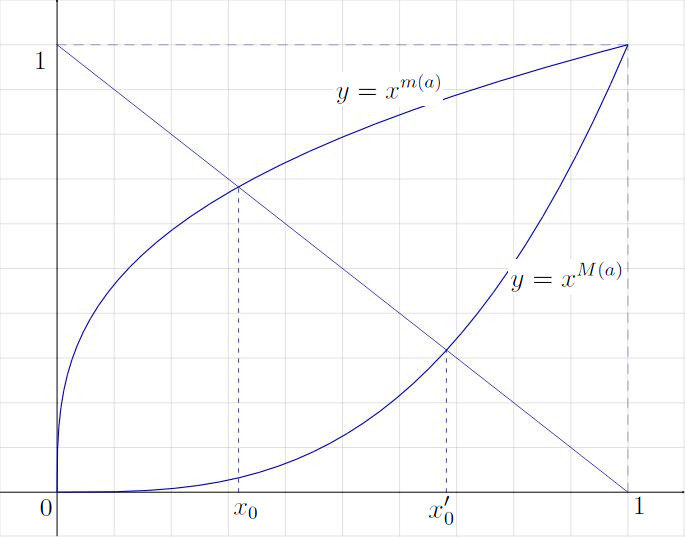}
 \end{center}

 Let $x_0$ be the unique number in $(0, 1/2)$ such that $x_0^{m(a)} = 1 - x_0$, 
 and $x'_0$ be the unique number in $(1/2, 1)$ such that $x'_0{}^{M(a)} = 1 - x'_0$.  
 Clearly 
 \begin{equation*}
  \lim_{a \to \infty} x_0 = 0, \quad
  \lim_{a \to \infty} x_0^{m(a)} = 1, \quad
  \lim_{a \to \infty} x_0^{m(a)/M(a)} = 1, \quad
  \lim_{a \to \infty} x'_0 = 1, \quad
  \textrm{and } \;
  \lim_{a \to \infty} (x'_0)^{M(a)} = 0. 
 \end{equation*}
 
 \begin{center}
  \includegraphics[scale=0.5]{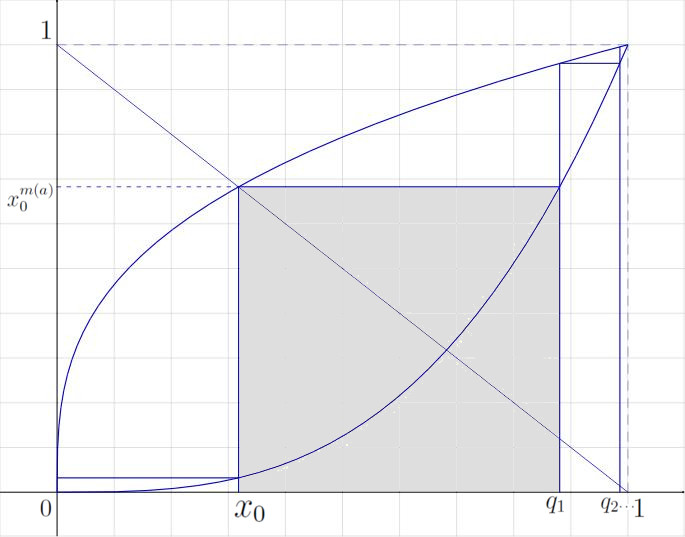}
 \end{center}

 On the other hand, let $q_1^{M(a)} = x_0^{m(a)}$, or $q_1 = x_0^{m(a)/M(a)}$,
 and in general $q_r^{M(a)} = q_{r-1}^{m(a)}$, or $q_r = x_0^{r m(a)/M(a)}$.  
 Then the rectangles $R_n$ in the figure contain the rectangle $[x_0, x_0^{m(a)/M(a)}] \times [0, x^{m(a)}]$, so
 \begin{equation*}
  (x_0^{m(a)/M(a)} - x_0) x_0^{m(a)} 
  < \sum_{n \geq 0} q_r^{a^{kn} M(a)} \left( q_r^{a^{nk}} - q_r^{a^{nk+k}} \right)
 \end{equation*}
 for $r = 1, 2, \ldots$, and $\lim_{r \to \infty} q_r = 1$.  
 
 \begin{center}
  \includegraphics[scale=0.5]{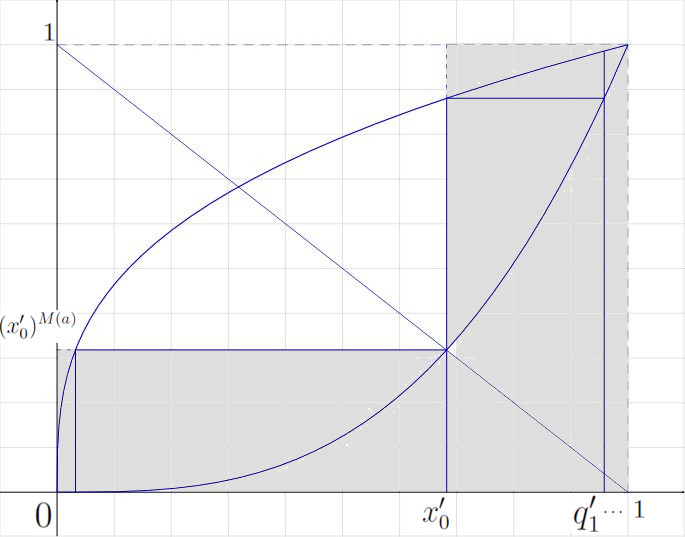}
 \end{center}

 Similarly, let $(q'_1)^{M(a)} = (x'_0)^{m(a)}$, or $q'_1 = (x'_0)^{m(a)/M(a)}$,
 and in general $(q'_r)^{M(a)} = (q'_{r-1})^{m(a)}$, or $q'_r = (x'_0)^{r m(a)/M(a)}$.  
 Then the rectangles $R_n$ in the figure are contained 
 in the gnomon $[0, 1] \times [0, 1] \backslash [ 0, x'_0) \times ( (x'_0)^{M(a)}, 1]$, so
 \begin{equation*}
  \sum_{n \geq 0} (q'_r)^{a^{kn} M(a)} \left( (q'_r)^{a^{nk}} - (q'_r)^{a^{nk+k}} \right) 
  < 1 - x'_0(1 - (x'_0)^{M(a)})
 \end{equation*}
 for $r = 1, 2, \ldots$, and $\lim_{r \to \infty} q'_r = 1$.  
 
 Finally, choose $a$ large ehough so that 
 \begin{equation}
 \label{eqExpoExpoKillerIneq}
  1 - x'_0(1 - (x'_0)^{M(a)}) < (x_0^{m(a)/M(a)} - x_0) x_0^{m(a)}.  
 \end{equation}
 This means there are two subsequences $\{q_r\}$ and $\{q'_r\}$
 in $(0, 1)$ both converging to 1 and yielding distinct cluster points
 for the sum
 \begin{equation*}
  \sum_{n \geq 0} q^{a^{kn} M(a)} \left( q^{a^{nk}} - q^{a^{nk+k}} \right).  
 \end{equation*}
 Consequently, 
 \begin{equation*}
  \lim_{q \to 1^{-}} \sum_{n \geq 0} q^{a^{kn} M(a)} \left( q^{a^{nk}} - q^{a^{nk+k}} \right)
 \end{equation*}
 cannot exist for large enuogh $a$.
 
 Now, because $j$ belongs to a finite set, 
 we can make the inequality \eqref{eqExpoExpoKillerIneq} work for all $j$, 
 by suitable selection of $a$.  
 Moreover, we can make the smaller side as close as we like to zero, 
 and the greater side as close as we like to 1.  
 
 Recall the decomposition 
 (after the change of variable $q \leftarrow q^{(a^k - 1)/(a^{k-1} - a^{j})}$)
 \begin{equation*}
 \sum_{j = 0}^{k-2} C(j) \sum_{n \geq 0} q^{a^{kn} M(a)} \left( q^{a^{nk}} - q^{a^{nk+k}} \right).  
 \end{equation*}
 We have shown that the inner sum oscillates between two values, 
 one close to zero, and the other close to 1.  
 Without loss of generality, we can assume that $C(k-1) \neq 0$, 
 so that $\displaystyle \sum_{j = 0}^{k-2} C(j) \neq 0$.  
 This can be achieved by taking the non-zero $C(\tilde{k})$ with the largest index.  
 Therefore, values of the ultimate double series oscillate between 
 zero and $C(k-1)$ as $q \to 1^-$ for large enough $a$.  
 This concludes the proof.  
\end{proof}

The series in Theorem \ref{thmPeriodicExpoExpo} is an instance of 
\emph{lacunary series} or \emph{Mahler function}.  
Theorem \ref{thmPeriodicExpoExpo} is true for any $a$, 
as shown in \cite{ZudilinEtAl}, 
where a much deeper account of the behavior of those series is given.  
The elementary approach here provides a visual aid to understanding the ``periodicity" in the limit.  

%

\section{Radial Limits}
\label{secRadial}

Given any real function $s(t)$ with $\lim_{t \to \infty} s(t) = \infty$, 
and a periodic function $C:\mathbb{N} \to \mathbb{C}$, 
The series
\begin{equation*}
  \sum_{n \geq 0} C(n) q^{s(n)}
\end{equation*}
converges for $\vert q \vert < 1$ thanks to the ratio test.  

\begin{center}
 \includegraphics[scale=0.7]{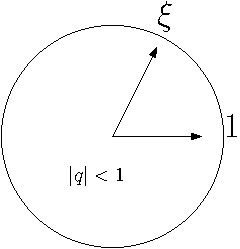}
\end{center}

If $\xi$ is any root of unity, we can take $q$ such that $q/\xi \in (0, 1)$
and consider
\begin{equation*}
  \lim_{q/\xi \to 1^{-}} \sum_{n \geq 0} C(n) q^{s(n)}, 
\end{equation*}
known as a radial limit.  
It reduces to 
\begin{equation*}
  \lim_{q \to 1^{-}} \sum_{n \geq 0} C(n) \, \xi^{s(n)} \, q^{s(n)}
\end{equation*}
after a change of parameter making $q \in (0, 1)$.  

When $s(n)$ is a polynomial, $\widetilde{C}(n) = C(n) \, \xi^{s(n)}$ 
is still a periodic function with possibly a different period, 
and it does not have to have mean zero.  
If $\widetilde{C}(n)$ has mean zero, then Theorem \ref{thmPeriodicMeanZeroCoeffs} applies.  
For instance, 
\begin{equation*}
  \lim_{q/\omega \to 1^{-}} \sum_{n \geq 0} q^{n^3}
  = \lim_{q \to 1^{-}} \sum_{n \geq 0} \omega^n \, q^{n^3}
  = \lim_{q \to 1^{-}} 1 + \omega q + \omega^2 q^8 + q^{27} + \omega q^{64} + \cdots
  = - \omega - \frac{2}{3}, 
\end{equation*}
where $\omega$ is a third root of unity, or $1 + \omega + \omega^2 = 0$.  
But, 
\begin{equation*}
  \lim_{q/\xi \to 1^{-}} \sum_{n \geq 0} (-1)^n q^{n^2}
  = \lim_{q \to 1^{-}} \sum_{n \geq 0} (-1)^n \, \xi^n \, q^{n^2}
  = \lim_{q \to 1^{-}} 1 - \xi q - \xi q^4 + q^{9} - \xi q^{16} - \xi q^{25} + \cdots
\end{equation*}
diverges for $\xi$ a primitive sixth root of unity.  

\section{Asymptotic Expansions}
\label{secAsymp}

We will use Euler's integral formula as given in Berndt's book \cite[Ch. 13, eq. (10.5)]{RNBII}.
\begin{align}
\label{eqEulerIntegralFormulaFinite}
  \sum_{n = a}^b f(n) = \int_a^b f(t) \, \mathrm{d}\,t \; 
  + \frac{f(a) + f(b)}{2}
  + \sum_{k=1}^{m-1} \frac{ B_{2k} }{ (2k)! } 
    \left\{ f^{(2k-1)}(b) - f^{(2k-1)}(a) \right\} + R_m, 
\end{align}
where
\begin{align*}
  R_m = \int_a^b \frac{ B_{2m} - B_{2m}(t - \lfloor t \rfloor) }{ (2m)! } f^{(2m)}(t) \, \mathrm{d} \, t, 
\end{align*}
$B_j$ are the Bernoulli numbers, and $B_j(t)$ are the Bernoulli polynomials \cite[p. 264]{ApANT}, 
and $\lfloor t \rfloor$ gives the integer part of $t$.  

For us, $f(n) = q^{s(n)}$ where $q \in (0, 1)$ and $s(n)$ is a real polynomial 
with positive degree $d$ and positive leading coefficient.  

Using Stirling's formula \cite[Sec. 1.4]{AAR}
\begin{align*}
  n! \sim \sqrt{2 \pi n} \left( \frac{n}{e} \right)^n, 
\end{align*}
as $n \to \infty$, 
and estimates for the Bernoulli numbers and the Bernoulli polynomials when $t \in [0, 1]$
\begin{align*}
  \vert B_{2m} \vert \sim 4 \sqrt{\pi m} \left( \frac{m}{\pi e} \right)^{2m}, \qquad
  \vert B_{2m}(t) \vert \leq \frac{2 \zeta(2m) (2m)! }{ (2 \pi)^{2m} }
\end{align*}
as $m \to \infty$, and for all $m$, respectively \cite[Theorem 12.18]{ApANT}, \cite{Lehmer}; 
where $\zeta(\cdot)$ is Riemann Zeta function, 
we deduce that 
\begin{align*}
  \vert R_m \vert \stackrel{<}{\sim} 
  \left\vert \left( \frac{2 + 2\zeta(2m) }{ (2 \pi)^{2m} } \right)
    \int_a^b f^{(2m)}(t) \, \mathrm{d} \, t \right\vert.  
\end{align*}
Notice that $\zeta(2m)$ is bounded as $m$ grows.  
There is a more precise bound for the maximum value of the absolute value of Bernoulli polynomials
inside the interval $[0, 1]$ \cite{Lehmer}, but the above reasoning is enough for our purposes.  

When $s(t)$ is a real polynomial with positive degree $d$ and positive leading coefficient, 
\begin{align*}
  f(t) = & q^{s(t)} = e^{(\log q)s(t)} \\
  f'(t) = & \left[ (\log q) s'(t) \right] f(t) \\
  f''(t) = & \left[ (\log q) s''(t) + (\log q)^2 (s'(t))^2 \right] f(t) \\
  f'''(t) = & \left[ (\log q) s'''(t) + 3 (\log q)^2 s'(t) s''(t) + (\log q)^3 (s'(t))^3 \right] f(t) \\
  & \vdots \\
  f^{(k)}(t) = & \left[ \sum_{r = 1}^k (\log q)^r 
    \sum_{c_1 + \cdots + c_r = k} a(c_1, \ldots, c_r) \prod_{j = 1}^r s^{(c_j)}(t) \right] f(t).  
\end{align*}
In the last expression, $c_j$'s are positive integers, 
and $a(c_1, \ldots, c_r)$ are integral coefficients that can be computed upon wish.  

When $k$ is large enough, $c_j > \lfloor \frac{k}{r} \rfloor$ for at least one $j$ in the inner sum, 
so that when $\lfloor \frac{k}{r} \rfloor > d$, the degree of $s(t)$, 
that term vanishes because $s^{(c_j)}(t) = 0$ identically.  
Thus, as $q \to 1^{-}$, $f^{(k)}(t) = O( (-\log q)^{r_0} )$, 
where $r_0 = \min_{r \in \mathbb{N}} \left\{ \lfloor \frac{k}{r} \rfloor \leq d \right\}$.  
This means, as $k \to \infty$, $r_0 \to \infty$ also.  

Finally, for $s(t)$ as above, and $p(t)$ any polynomial, 
it is straightforward to see that the intregral 
\begin{align*}
  \int_0^\infty q^{s(t)} p(t) \, \mathrm{d} \, t
\end{align*}
converges for any $q \in (0, 1)$.  
Thus, under the said assumptions, 
\begin{align*}
  \sum_{n \geq 0} q^{s(n)}, \qquad 
  \int_0^\infty q^{s(t)} \, \mathrm{d} t, \qquad
  \int_0^\infty \left( \frac{ \mathrm{d} }{ \mathrm{d} t } \right)^M q^{s(t)} \, \mathrm{d} t
\end{align*}
all converge.  
Therefore, Euler's integral formula \eqref{eqEulerIntegralFormulaFinite}
in our context remains valid when $a = 0$ and $b \to \infty$, and becomes
\begin{align}
  \nonumber
  \sum_{n \geq 0} q^{s(n)} 
  = & \int_0^\infty q^{s(t)} \,\mathrm{d}\,t\,
  + \frac{q^{s(0)}}{2} 
  - \sum_{k = 1}^{m-1} \frac{ B_{2k} }{ (2k)! } 
    \left[ \left( \frac{ \mathrm{d} }{ \mathrm{d} \, t } \right)^{2k-1} q^{s(t)} \right]_{t = 0} \\ 
\label{eqEulerIntegralFormulaInfinite}
  & + \int_0^\infty \frac{ B_{2m} - B_{2m}(t - \lfloor t \rfloor) }{ (2m)! }  
    \left[ \left( \frac{ \mathrm{d} }{ \mathrm{d} \, t } \right)^{2m} q^{s(t)} \right] \,\mathrm{d}\,t\,
\end{align}

In \eqref{eqEulerIntegralFormulaInfinite}, the fraction $q^{s(0)}/2$ and the finite sum on the right hand side
give series in integral powers of $(-\log q)$.  
The coefficients are calculable upon wish.  
The last integral on the right hand side, by the preceding discussion, 
is $O( (-\log q)^{r_0} )$ where $r_0 \to \infty$ as $m \to \infty$.  

As for the first integral on the right hand side of \eqref{eqEulerIntegralFormulaInfinite}, 
we write
\begin{align*}
  s(t) = a_d t^d + a_{d-1} t^{d-1} + \cdots + a_1 t + a_0
\end{align*}
for $d > 0$ and $a_d > 0$.  
\begin{align*}
  \int_{0}^\infty q^{s(t)} \,\mathrm{d}\,t\, 
  = \int_0^\infty e^{-(-\log q) a_d t^d} 
    \; e^{-(-\log q) [a_{d-1} t^{d-1} + \cdots + a_1 t + a_0] } \,\mathrm{d}\,t\,
\end{align*}
after the change of variable 
\begin{align*}
  u = (-\log q) a_d t^d, \qquad
  t = \frac{ u^{1/d} }{ (-\log q)^{1/d} a_d^{1/d} }, \qquad
  \mathrm{d} \, t = \frac{ u^{(1 - d)/d} }{ d (-\log q)^{1/d} a_d^{1/d} } \, \mathrm{d} \, u
\end{align*}
becomes
\begin{align*}
  & \frac{1}{ d (-\log q)^{1/d} a_d^{1/d} }
  \int_0^\infty u^{(1 - d)/d} e^{-u} \; 
    e^{ \tilde{a}_{d-1} (-\log q)^{1/d} u^{(d-1)/d} + \cdots 
      + \tilde{a}_1 (-\log q)^{(d-1)/d} u^{1/d} + \tilde{a_0} (-\log q) } \; \mathrm{d} \, u \\
  & = \frac{1}{ d (-\log q)^{1/d} a_d^{1/d} }
  \int_0^\infty u^{(1 - d)/d} e^{-u} \; \\
  & \sum_{j \geq 0} \frac{ \left[\tilde{a}_{d-1} (-\log q)^{1/d} u^{(d-1)/d} + \cdots 
    + \tilde{a}_1 (-\log q)^{(d-1)/d} u^{1/d} + \tilde{a_0} (-\log q) \right]^j }{ j! }
  \; \mathrm{d} \, u \\
  & = \frac{1}{ d (-\log q)^{1/d} a_d^{1/d} } 
  \left[ \Gamma(1/d) + c_0 (-\log q)^{1/d} + c_1 (-\log q)^{2/d} + \cdots \right] 
\end{align*}
where $\tilde{a}_j = -a_j / a_d^{j/d} $ for $j = 0, 1, \ldots, d-1$.  

The above computations, together with Theorem \ref{thmPeriodicMeanZeroCoeffs} leads to the following.  

\begin{conj}
\label{conjAsympExp}
  Suppose $C : \mathbb{N} \to \mathbb{C}$ is a periodic function with period $k$.  
  Let $s(t)$ be a real polynomial with positive degree $d$ and positive leading coefficient $a_d$.  
  Then, 
  \begin{align*}
    & \sum_{n \geq 0} C(n) q^{s(n)} \sim \\
    & \begin{cases}
      \frac{-C(1) - 2C(2) - \cdots - (k-1)C(k-1)}{k} + c_1 (-\log q)^{1/d} + c_2 (-\log q)^{2/d} + \cdots
	& \textrm{if } C \textrm{ has mean zero},  \\
      \phantom{0} & \phantom{0} \\
      \frac{ (C(1) + \cdots + C(k)) \Gamma(1/d) }{ d (-\log q)^{1/d} a_d^{1/d} } + c_0 + c_1 (-\log q)^{1/d} + c_2 (-\log q)^{2/d} + \cdots
	& \textrm{otherwise}.  
    \end{cases}
  \end{align*}
  The coefficients $c_j$ are $\mathbb{Q}$-linear combinations of integers, 
  values of the Gamma function at various fractions $l/d$, and Bernoulli numbers $B_{2k}$.  
\end{conj}

One should note that the coefficients $c_j$ are not necessarily the same for the separate cases in the conjecture. 

When $s(n)$ has degree greater than 2, or when $C(n)$ does not have mean zero, 
Euler's integral formula indicates that the non-integral powers of $(-\log q)$ persist.  

Some specific cases of the conjecture are proven in literature.  
For instance, if $C(n)$ has mean zero, and $s(n)$ is a certain linear or quadratic polynomial, 
the expansion is given in \cite[p.98]{LZ}, 
where the half integral powers vanish.  
Indeed, Euler's formula also suggests that when $s(n)$ is an arbitrary quadratic polynomial 
and $C(n)$ has mean zero, the half integral powers vanish and we get a series 
in integral powers of $(- \log q)$.  
Ramanujan has a slightly different formula, 
a more general series where $C(n)$ is a polynomial, 
but $s(n)$ is a monomial with arbitrary degree \cite[Ch 15, Theorem 3.1]{RNBII}.  
There, also, the non-integral powers of $(-\log q)$ vanish, 
in accordance with Euler's integral formula.  
Again, we failed to find the $s(n) = $ an arbitraty polynomial case in the literature.  

One can try the Mellin transform approach in \cite[Ch. 15]{RNBII}.  
For convenience, we set $x = -\log q$, 
and we assume that $s(t)$ is a real polynomial with positive degree, 
yielding positive values for all non-negative numbers.  
In the integral, 
\begin{align*}
  \int_0^{\infty} \sum_{n \geq 0} e^{-x s(n)} x^{\sigma -1} \, \mathrm{d} \, x
\end{align*}
notice that the inner sum is uniformly convergent, so that we can switch the order of summation and integration.  
Upon the change of variable $u = s(n) x$ for each $n$ in the inner sum gives
\begin{align*}
  \sum_{n \geq 0} \frac{1}{s(n)^\sigma} \int_0^\infty e^{-u} u^{\sigma - 1} \, \mathrm{d} \, u
  = \Gamma(\sigma) \sum_{n \geq 0} \frac{1}{s(n)^\sigma}.  
\end{align*}
Let $\xi_1, \ldots, \xi_d$ be the roots of $s(n)$.  
Using the fundamental theorem of algebra and fractional decomposition, 
one gets a linear combination of Hurwitz zeta functions $\zeta(\sigma, \xi_j)$
(possibly $\zeta(2\sigma, \xi_j)$, $\zeta(3\sigma, \xi_j)$, etc. depending on the multiplicity of the roots).   
\begin{align*}
  \sum_{n \geq 0} e^{-x s(n)} 
  = \frac{1}{2 \pi i} \int_{a-i\infty}^{a+i\infty} \Gamma(\sigma) x^{-\sigma}
    \sum_{j=1}^d \gamma_j \zeta(b_j \sigma, \xi_j) \, \mathrm{d} \, \sigma
\end{align*}
Here, $\gamma_j$'s are complex numbers, and $b_j$ are positive integers.  
The residue threorem seems to yield an asymptotic expansion in terms of integral exponents of $x$.  
There is an apparent mismatch between the suggestion of Euler's integral formula, and Mellin transform approach.  
We therefore leave Conjecture \ref{conjAsympExp} as unsettled.  

\section*{Acknowledgements}

The author owes sincere thanks for the useful discussions in the preparation of this manuscript to
George E. Andrews, Robert C. Rhoades, Bruce C. Berndt, 
Krishnaswami Alladi, Alexander Berkovich, Peter Paule, 
and Cem Y. Y{\i}ld{\i}r{\i}m.  

The author is also indebted to Wadim Zudilin for immediately pointing out \cite{ZudilinEtAl}
with some other references, 
and to the anonymous referee for making the proofs flawless.

\bibliographystyle{amsplain}

\begin{thebibliography}{normalsize}

\bibitem{AAR}
G. E. Andrews, R. Askey, R. Roy, 
\emph{Special Functions}, 
Encyclopedia of Mathematics and its Applications, Vol. 71, 
Cambridge University Press, U.S., 2006.  

\bibitem{ApANT}
T. M. Apostol, 
\emph{Introduction to Analytic Number Theory}, 
Springer, U.S., 1976.  

\bibitem{BMO}
A. Berkovich, B. M. McCoy, W. P. Orrick, 
Polynomial identities, indices, and duality for the $N = 1$ superconformal model $SM (2, 4v)$, 
\emph{Journal of Statistical Physics}, {\bf 83(5-6)}, pp. 795--837, 1996.  

\bibitem{RNBII}
B. C. Berndt, 
\emph{Ramanujan's Notebooks, Part II}, 
Springer, U.S., 1999.  

\bibitem{ZudilinEtAl} R. P. Brent, M. Coons, W. Zudilin, 
Algebraic Independence of Mahler Functions via Radial Asymptotics, 
\emph{ Int Math Res Notices}, published online May 2015.  


\bibitem{CG}
J. H. Conway, R. Guy, 
\emph{The Book of Numbers}, 
Springer, 1996.  

\bibitem{GR}
G. Gasper, M. Rahman, 
\emph{Basic Hypergeometric Series, second edition}, 
Encyclopedia of Mathematics and its Applications Vol. 96, 
Cambridge University Press, U.K., 2004.  

\bibitem{LZ}
R. Lawrence, D. Zagier, 
Modular forms and quantum invariants of 3-manifolds, 
\emph{Asian J. Math}, {\bf 3(1)}, pp. 93--108, 1999.  

\bibitem{Lehmer}
D. H. Lehmer, 
On the Maxima and Minima of Bernoulli Polynomials, 
\emph{The American Mathematical Monthly}, {\bf 47(8)}, pp. 533--538, 1940.  

\bibitem{Rudin}
W. Rudin, 
\emph{Real and Complex Analysis, third edition}, 
McGraw-Hill, U.S., 1987.  

\bibitem{Stewart}
J. Stewart, 
\emph{Calculus, seventh edition}, 
Brooks/Cole, U.S., 2012.  

\end{thebibliography}

\end{document}